\documentclass[12pt]{article}
\usepackage{amsmath, amsfonts, latexsym, amssymb, graphicx, mathtools}
\title{ Semistability and Simple Connectivity at Infinity of Finitely Generated Groups with a Finite Series of Commensurated Subgroups}
\author{Michael Mihalik}
\newtheorem{theorem}{Theorem}[section]

\newtheorem{lemma}[theorem]{Lemma}

\newcounter{remarknum}
\newenvironment{remark}{\addvspace{12pt}\refstepcounter{remarknum}
\noindent{\bf Remark \arabic{remarknum}.}}{\par\addvspace{12pt}}

\newcounter{claimnum}
\newenvironment{claim}{\addvspace{12pt}\refstepcounter{claimnum}
\noindent{\bf Claim \arabic{claimnum}.}}{\par\addvspace{12pt}}

\newcounter{definitionnum}
\newenvironment{definition}{\addvspace{12pt}\refstepcounter{definitionnum}
\noindent{\bf Definition \arabic{definitionnum}.}}{\par\addvspace{12pt}}

\newenvironment{proof}{\addvspace{12pt}\noindent{\bf Proof:}}{
$\Box$\par\addvspace{12pt}}
\newcounter{examplenum}
\newenvironment{example}{\addvspace{12pt}\refstepcounter{examplenum}
\noindent{\bf Example \arabic{examplenum}.}}{\par\addvspace{12pt}}
\date{November 3, 2014}

\begin{document}  
\maketitle
\begin{abstract} 
 A subgroup $H$ of a group $G$ is {\it commensurated} in $G$ if for each $g\in G$, $gHg^{-1}\cap H$ has finite index in both $H$ and $gHg^{-1}$. If there is a sequence of subgroups $H=Q_0\prec Q_1\prec \cdots \prec Q_{k}\prec Q_{k+1}=G$
 where $Q_i$ is commensurated in $Q_{i+1}$ for all $i$, then $Q_0$ is {\it subcommensurated} in $G$. In this paper we introduce the notion of  the simple connectivity at infinity of a finitely generated group (in analogy with that for finitely presented groups).  Our main result is: If a finitely generated group $G$ contains an infinite, finitely generated, subcommensurated  subgroup $H$, of infinite index in $G$, then $G$ is 1-ended and semistable at $\infty$. If additionally, $H$ is finitely presented and 1-ended, then $G$ is simply connected at 
 $\infty$. A normal subgroup of a group is commensurated, so this result is a strict generalization of a number of results, including the main theorems of G. Conner and M. Mihalik \cite{CM}, B. Jackson \cite{J},  V. M. Lew \cite{L}, M. Mihalik \cite{M1}and \cite{M2}, and J. Profio \cite{P}.
\end{abstract}

\section{Introduction and Background}

In 1962, J. Stallings defined what it means for a space to be $n$-connected at $\infty$, and proved the following:

\begin{theorem} (J. Stallings \cite{St}) If $V^n$, $n\geq 5$, is a contractible PL $n$-manifold without boundary, then $V$ is PL-homeomorphic to $R^n$ if and only if $V$ is simply connected at $\infty$. 
\end{theorem}

In 1974,   R. Lee and F. Raymond first considered the fundamental group of an end of a group. In particular, they considered groups that are simply connected at $\infty$. 

\begin{theorem} (R. Lee, F. Raymond \cite{LR})\label{LR} Let $G$ be a finitely presented group with normal subgroup $N$ isomorphic to $\mathbb Z^k$ and quotient $K= G/N$. Assume when $k=1$ that $K$ is 1-ended and that when $k=2$ that $K$ is not finite, and no restrictions when $k>2$. Then $G$ is simply connected at $\infty$. 
\end{theorem}

For a reasonable space $X$ (or finitely presented group $G$), one needs to know that $X$ (respectively $G$) is semistable at $\infty$ in order to have the fundamental group of an end of $X$ (respectively $G$) defined independent of base ray. In 1982, B. Jackson generalized Theorem \ref{LR} and in 1983, M. Mihalik proved the first semistability at $\infty$ theorem for a class of finitely presented groups. These two  results serve as a starting point for this paper. 

\begin{theorem}  \label{J}  (B. Jackson \cite{J}) 
If $H$ is an infinite, finitely presented, normal subgroup of infinite index in the finitely presented group $G$, and either $H$ or $G/H$ is 1-ended. Then $G$ is simply connected at $\infty$.   
\end{theorem}

\begin{theorem} \label{M1} (M. Mihalik \cite{M1}) 
If $H$ is an infinite, finitely generated, normal subgroup of infinite index in the finitely presented group $G$, then $G$ is semistable at $\infty$. 
\end{theorem}
 
In 1985, the following connections were drawn between semistability and simple connectivity at $\infty$, and group cohomology.

\begin{theorem} (R. Geoghegan, M. Mihalik \cite{GM}) \label{GM} 
If $G$ is a finitely presented and  semistable at $\infty$ group then $H^2(G,\mathbb ZG)$ is free abelian. If $G$ is simply connected at $\infty$ then $H^2(G,\mathbb ZG)=0$. 
\end{theorem}

It is unknown whether or not all finitely presented groups are semistable at $\infty$. It is also unknown whether or not for all finitely presented groups $G$, $H^2(G,\mathbb ZG)$ is free abelian. The main theorem in the unpublished 1993 PhD dissertation of V. Ming Lew generalized Theorem \ref{M1} and the main theorem of the 1990 PhD dissertation of J. Profio generalized Theorem \ref{J}: 

\begin{theorem}\label{L}
(V. M. Lew \cite{L}) 
Suppose $H$ is an infinite,  finitely generated, subnormal subgroup of the finitely generated group $G$: 
$$H = N_0 \lhd N_1 \lhd N_2 \lhd \ldots  \lhd N_k = G, \hbox{ for } k\geq 1$$
and $H$ has infinite index in $G$. Then $G$ is 1-ended and semistable at $\infty$. 
\end{theorem}

\begin{theorem} \label{P}
(J. Profio \cite{P}) 
Suppose $H\lhd N\lhd G$ is a normal series with $H$ and $G$ finitely presented, and $H$ 1-ended and of infinite index in $G$. Then $G$ is simply connected at $\infty$.
\end{theorem}

Given a subgroup $H$ of a group $G$, the element $g\in G$ is in the {\it commensurator} of $H$ in $G$ (denoted $Comm(H,G)$) if $gHg^{-1}\cap H$ has finite index in both $H$ and $gHg^{-1}$. The subgroup $H$ is {\it commensurated} in  $G$ if $Comm(H,G)=G$, so normal subgroups are commensurated. The main result of \cite{CM} generalizes Theorems \ref{M1} and \ref{J} in a direction different than these last two results:

\begin{theorem}\label{MainCM} (G. Conner, M. Mihalik \cite{CM}) 
If a finitely generated group $G$ has an infinite, finitely generated, commensurated subgroup $Q$, and $Q$ has infinite index in $G$, then $G$ is 1-ended and semistable at $\infty$. Furthermore, if $G$ and $Q$ are finitely presented and either $Q$ is 1-ended or the pair $(G,Q)$ has one filtered end, then $G$ is simply connected at $\infty$.
\end{theorem}

\begin{example} 
For $p$ a prime, the  group $SL_n(\mathbb Z[{1\over p}])$ is finitely presented. When $n>2$ the only normal subgroups of this group are either finite or of finite index (see \cite{SW}). For $n>2$, the finitely presented 1-ended subgroup, $SL_n(\mathbb Z)$ is commensurated in $SL_n(\mathbb Z[{1\over p}])$ and so by Theorem \ref{MainCM},  $SL_n(\mathbb Z[{1\over p}])$ is 1-ended and simply connected at $\infty$. 
\end{example}

While Lew's theorem improved Theorem \ref{M1} by replacing normality by subnormality, Profio's result was the best attempt  in the last 30 years, to improve the normality hypothesis of Theorem \ref{J} to subnormality. As a corollary of our main theorem, we obtain the subnormal version of Jackson's Theorem \ref{J}. The semistability part of Theorem \ref{MainA} is proved first and then used in an essential way in the proof of the simply connected at $\infty$ part of Theorem \ref{MainA}. A new idea, the simple connectivity at $\infty$ of a finitely generated group, is introduced and used in a fundamental way to prove the second part of Theorem \ref{MainA}.  We point out that we cannot prove this part of Theorem \ref{MainA}, even in the finitely presented case, without this new concept.

If $Q$ is a commensurated subgroup of $G$ we use the notation $Q\prec G$.
The main theorem of this article is the following:

\begin{theorem} \label{MainA} (Main Theorem)
Suppose $H$ is a finitely generated infinite subgroup of infinite index in the finitely generated group $G$, and  $H$ is subcommensurated in $G$:
$$H=Q_0\prec Q_1\prec \cdots \prec Q_{k}\prec G$$
Then $G$ is 1-ended and semistable at infinity. If additionally, $H$ is 1-ended and finitely presented then the finitely generated group $G$ is simply connected at $\infty$.
\end{theorem}  

In the next section we define what it means for a finitely generated group to be simply connected at $\infty$ (a strict generalization of simple connectivity at $\infty$ for finitely presented groups).

\begin{example}
In \cite{M5}, short exact sequences are produced for each $n>0$, of the form:
$$ 1\to H\to (\mathbb Z^n\ast \mathbb Z)\times (\mathbb Z^n\ast \mathbb Z)\to\mathbb Z^n\to 1$$
where $H$ is 1-ended and finitely generated. The group $(\mathbb Z^n\ast \mathbb Z)\times (\mathbb Z^n\ast \mathbb Z)$ is not simply connected at $\infty$ and the group $Z^n$ is $(n-2)$-connected at $\infty$. 
These elementary examples shows that the finitely presented hypothesis on $H$ in Theorems \ref{J} and \ref{MainA}, cannot be relaxed.
\end{example}

The remainder of the paper is organized as follows: In $\S 2$ the working definitions and notation are established. We introduce our definition of a subgroup being simply connected at infinity inside an overgroup. This definition is then used to define the simple connectivity at $\infty$ of a finitely generated group. We end $\S 2$ with an important technical lemma.

In $\S 3$, we prove the semistability part of our main theorem. This is an induction argument that starts with base case given by Theorem \ref{MainCM}. 

In $\S4$, we prove the simply connectivity at $\infty$ part of our main theorem. This is also an induction argument that starts with the base case given by the simple connectivity part of Theorem \ref{MainCM}.  The semistability result of $\S 3$ is used in conjunction with Lemma  \ref{equiv} to set up the proof of the simple connectivity part of Theorem \ref{MainA}.

\section{Definitions and a Technical Lemma}
R. Geoghegan's book \cite {G} is a general reference to all that is in this section. A continuous function $f:X\to Y$ is {\it proper} if for each compact subset $C$ of $Y$, $f^{-1}(C)$ is compact in $X$. A proper map $r:[0,\infty)\to X$ is called a {\it ray} in $X$.  If $K$ is a locally finite, connected CW-complex, then one can define an equivalence
relation $\sim$ on the set $A$ of all rays in $ K$ by setting $r \sim s$ if and only if
for each compact set $C \subset K$, there exists an integer $N(C)$ such that $r([N(C),\infty))$ and
$s([N(C), \infty))$ are contained in the same unbounded path component of $K -C$ (a path
component of $K-C$ is {\it unbounded} if it is not contained in any compact subset of $K$). An equivalence class of $A/\sim$ is called {\it an end of} $K$, the set of equivalence classes of $A/\sim$ is called {\it the set of ends of} $K$ and two rays in $K$, in the same equivalence class, are said to {\it converge to the same end}.
The cardinality of $A/\sim$, denoted by $e(K)$, is the {\it number of ends of} $K$.

If G is a finitely generated group with generating set $\mathcal S$, then the {\it Cayley graph of $G$ with respect to $\mathcal S$}, denoted $\Gamma_{(G,\mathcal S)}$, has vertex set $G$ and an edge between vertices $v$ and $w$ if $vs=w$ for some $s\in\mathcal S$. We define the {\it number of ends of $G$}, denoted by $e(G)$, to be the number of ends of the Cayley graph of $G$ with respect to a finite generating set.
(In particular, $e(G) = e(\Gamma_{(G,\mathcal S)}$). This definition is independent of the choice of finite generating set for $G$.
If $G$ is finitely generated, then $e(G)$ is either 0, 1, 2, or is infinite (in which case it has the
cardinality of the real numbers). We let $\ast$ denote the basepoint of $\Gamma_{(G,\mathcal S)}$, which corresponds to the identity of $G$.

If $f$ and  $g$ are rays in $K$, then one says that $f$ and $g$ are {\it properly homotopic} if there is a proper map
$H : [0,1] \times [0,\infty) \to K$ such that
$H\vert_{\{0\}\times[0,\infty)} = f$ and $H\vert_{\{1\}\times[0,\infty)} = g$. If $f(0)=g(0)=v$ and $H\vert _{[0,1]\times \{0\}}=v$, one says $f$ and $g$ are {\it properly homotopic relative to $v$} (or $rel \{v\}$).

\begin{definition} \label{defss} A locally finite, connected CW-complex $K$ is {\it semistable at $\infty$} if any two rays in $K$ converging to the same end are properly homotopic. The space $K$ is {\it simply connected at $\infty$} if for any compact set $C\subset K$ there is a compact $D\subset K$ such that loops in $K-D$ are homotopically trivial in $K-C$.
\end{definition}

In a locally finite CW complex, any ray is properly homotopic to an edge path ray.  So in order to show semistability in such a complex, it is enough to prove edge path rays converging to the same end are properly homotopic. 

Theorem 2.1 of  \cite{M1}, and  Lemma 9 of \cite{M2},  provide several equivalent notions of semistability.  The space considered in \cite{M1} is simply connected, but simple connectivity is not important in that argument. A slight modification of proofs give the following result. (See \cite{CM}.)

\begin{theorem}\label{ssequiv} 
Suppose $K$ is a locally finite, connected and 1-ended CW-complex. Then the following are equivalent:
\begin{enumerate}
\item $K$ is semistable at $\infty$.
\item For any ray $r:[0,\infty )\to K$ and compact set $C$, there is a compact set $D$ such that for any third compact set $E$ and loop $\alpha$ based on $r$ and with image in $K-D$, $\alpha$ is homotopic $rel\{r\}$ to a loop in $K-E$, by a homotopy with image in $K-C$.
\item For some (equivalently any) ray $r$ in $K$ and any collection of compact sets $C_i$ such that $\cup_{i=1}^\infty C_i=K$ and $C_{i-1}$ is a subset of the interior of $C_i$, the inverse system:
$$\pi_1(X-C_1,r)\leftarrow \pi_1(X-C_2,r)\leftarrow\ldots$$ 
with bonding maps induced by inclusion along $r$, is pro-isomorphic to an inverse system of groups with epimorphic bonding maps. 
\item For any compact set $C$ there is a compact set $D$ such that if $r$ and $s$ are rays based at $v$ and with image in $K-D$, then $r$ and $s$ are properly homotopic $rel\{v\}$, by a proper homotopy in $K-C$. 
\end{enumerate}
If $K$ is simply connected (or if a group acting by homeomorphisms on $K$, acts transitively on the vertices of $K$) then a fourth equivalent condition can be added to this list:

5. If $r$ and $s$ are rays based at $v$, then $r$ and $s$ are properly homotopic 

$\ \ \ $$rel\{v\}$. 
\end{theorem}
 
If finite connected CW complexes  $X$ and $Y$ have isomorphic fundamental groups,  then  the universal cover of $X$ is semistable (simply connected) at $\infty$ if and only if the universal cover of $Y$ is semistable (simply connected)  at $\infty$. This result can be seen from the early work of  F. E. A. Johnson \cite{FEA} and \cite{FEA2}, or the proof of Theorem 3 of \cite{LR}. For a complete argument see the first three sections of Chapter 5 of R. Geoghegan's book \cite{G}.

\begin{definition} \label{defss2}
If $G$ is a 1-ended, finitely presented group and, $X$ is some (equivalently any) finite, CW-complex with fundamental group $G$, then we say $G$ {\it  is semistable at $\infty$} if the universal cover of $X$ is semistable at $\infty$. We say $G$ is {\it simply connected at $\infty$} if the universal cover of $X$ is simply connected at $\infty$. 
\end{definition}

The notion of semistabilty for a finitely generated group was first defined in \cite{M4}. We give
the definition for 1-ended groups since this is the case that concerns us. Suppose $G$ is a 1-ended finitely generated group with generating set $\mathcal S\coloneqq \{g_1, g_2,\ldots , g_n\}$ and let $\Gamma_{
(G,\mathcal S)}$ be the Cayley graph of $G$ with respect to this generating set.  Suppose $\{\alpha_1, \alpha_2,\ldots , \alpha_m\}$ is a finite set of relations in $G$ written in the letters $\{g_1^\pm, g_2^\pm,\ldots , g_n^\pm\}$.  For any vertex $v\in \Gamma_{(G,\mathcal S)}$, there is an edge path cycle labeled $\alpha_i$ at $v$.   The 2-dimensional CW-complex $\Gamma_{(G,\mathcal S)}(\alpha_1,\ldots , \alpha_m)$ is obtained by attaching to each vertex of $\Gamma_{(G,\mathcal S)}$, $2$-cells corresponding to the relations $\alpha_1,\ldots ,\alpha_n$.

We show in  \cite{M4}, that if $\mathcal S$ and $\mathcal T$ are finite generating sets for the group $G$ and there are finitely many $\mathcal S$-relations $P$ such that $\Gamma_{(G,\mathcal S)}(P)$ is semistable at $\infty$, then there are finitely many $\mathcal T$-relations $Q$ such that $\Gamma_{(G,\mathcal T)}(Q)$ is semistable at $\infty$. Hence the following definition:

\begin{definition} \label{defss3}
A finitely generated group  {\it $G$ is semistable at $\infty$} if  for some (equivalently any) finite generating set $\mathcal S$ for $G$ and finite set of $\mathcal S$-relations $P$ the complex $\Gamma_{(G,\mathcal S)}(P)$ is semistable at $\infty$. 
\end{definition}
      
Note that if $G$ has finite presentation $\langle \mathcal S:P\rangle$, then $G$ is semistable at $\infty$ with respect to Definition \ref{defss2} if and only if $G$ is semistable at $\infty$ with respect to Definition \ref{defss3} if and only if $\Gamma_{(G,\mathcal S)}(P)$ is semistable at $\infty$.

The following definition defines what it means for a finitely generated subgroup of a finitely presented group to be simply connected at $\infty$ relatively to the finitely presented over group. 

\begin{definition}\label{SCin} 
A finitely generated subgroup $A$ of a finitely presented group $G$ is {\it simply connected at $\infty$ in $G$ (or relative to $G$)} if for some (equivalently any by Lemma \ref{equiv} with $N=0$) finite presentation $\langle \mathcal A,\mathcal B;R\rangle$ of the group $G$  (where $\mathcal A$ generates $A$ and $\mathcal A\cup \mathcal B$ generates $G$), the 2-complex $\Gamma_{(G,\mathcal A\cup \mathcal B)}(R)$ has the following property:

Given any compact set $C\subset \Gamma_{(G,\mathcal A\cup \mathcal B)}(R)$ there is a compact set $D\subset \Gamma_{(G,\mathcal A\cup \mathcal B)}(R)$ such that any edge path loop in  $\Gamma_{(A,\mathcal A)}-D$ is homotopically trivial in $\Gamma_{(G,\mathcal A\cup \mathcal B)}(R)-C$. 
\end{definition}

In order to define what  it means for a finitely generated group $G$ to be simply connected at $\infty$, we must know that $G$ embeds in some finitely presented group. In 1961,G. Higmann proved:

\begin{theorem} (Higmann \cite{H}) A finitely generated infinite group $G$ can be embedded in a finitely presented group if and only if the set of relators of $G$ (as a set of freely reduced words in the generators) is recursive enumerable.
\end{theorem}

\begin{definition}\label{SCfg} 
A finitely generated and recursively presented group $A$ is {\it simply connected at $\infty$} if for any finitely presented group $G$ and subgroup $A'$ isomorphic to $A$, the subgroup $A'$ is simply connected at $\infty$ in $G$. 
\end{definition}

Suppose that $G$ is a finitely presented group and that $G$ satisfies the simply connected at $\infty$ condition of Definition \ref{defss2},
then $G$ satisfies Definition \ref{SCfg}, and there is no ambiguity. Futhermore, any finitely generated subgroup of $G$ is simply connected at $\infty$ in $G$. 

We conclude this section with Lemma \ref{equiv}, but first some terminology. Suppose $\langle \mathcal S:R\rangle$ is a finite presentation for a group $G$. If $A$ is a subcomplex of $\Gamma_{(G,\mathcal S)}(R)$, then $St(A)$ is the subcomplex of $\Gamma_{(G,\mathcal S)}(R)$ whose vertices $V(St(A))$ are the vertices of $A$ along with each vertex of $\Gamma_{(G,\mathcal S)}(R)$ that is connected to a vertex of $A$ by an edge. The edges $E(St(A))$ of $St(A)$ are all edges of $A$ and all edges of $\Gamma_{(G,\mathcal S)}(R)$, both of whose vertices are contained in $V(St(A))$. The 2-cells $F(St(A))$ of $St(A)$are all 2-cells of $A$ along with all 2-cells $F$, such that all vertices of $F$ belong to $V(St(A))$. 
If $A$ is an arbitrary subset of $\Gamma_{(G,\mathcal S)}(R)$ then let $\hat A$ be the smallest subcomplex of $\Gamma_{(G,\mathcal S)}(R)$ containing $A$ and define $St(A)$ to be $St(\hat A)$. 

\begin{lemma} \label{bump} 
(1) Suppose $A$ and $B$ are subcomplexes of $\Gamma_{(G,\mathcal S)}(R)$ and $St(A)\cap B\ne\emptyset$. Then $A\cap St(B)\ne\emptyset$. 

(2) Suppose $A$ is a subcomplex of $\Gamma_{(G,\mathcal S)}(R)$ and $B$ is an arbitrary subset of $\Gamma_{(G,\mathcal S)}(R)$ and $St(B)\cap A\ne\emptyset$ then $St^{L+1}(A)\cap B\ne\emptyset$ where $L$ is the length of the longest relation in $R$. 
\end{lemma}
\begin{proof}
Case (1). If $St(A)\cap B\ne\emptyset$ then there is a vertex $v\in St(A)\cap B$. If $v\in A$ then we are finished. Otherwise, there is a vertex $w\in A$ and an edge from $v$ to $w$. Then $w\in A\cap St(B)$

Case (2). Let $v$ be a vertex in $St(B)\cap A=St(\hat B)\cap A$. If $v\in B$ we are finished. Otherwise, $v\in \hat B$ or $v$ is adjacent to a vertex $w\in \hat B$. If $v\in \hat B$ then there is an edge $e$ containing a point $b\in B$ and  $v$ is  a vertex of $e$, or there is a 2-cell $F$ containing a point $b\in B$ and $v$ is a vertex of $F$. In either case, $b\in St^L(v)$, so $b\in B\cap St^L(A)$.  
If $v$ is adjacent to a vertex $w\in \hat B$ then as above, there is $b\in B\cap St^L(w)\subset B\cap St^{L+1}(A)$.
\end{proof}
The following technical lemma has a somewhat standard proof. 
\begin{lemma}\label{equiv} 
Suppose $A$ is a finitely generated subgroup of the finitely presented group $G$. Then $A$ is simply connected at $\infty$ in $G$ if and only if:

\noindent $(\dagger)$ For $\langle \mathcal S;R\rangle$ an arbitrary finite presentation for $G$, $N\geq 0$ an integer and $C$ a compact subset of $\Gamma_{(G,\mathcal S)}(R)$, there is a compact set $D(C,N)\subset \Gamma$ such that if $\alpha$ is an edge path loop in $\Gamma-D$ and each vertex of $\alpha$ is within $N$ of some vertex  of $A(\subset \Gamma)$ then $\alpha$ is homotopically trivial in $\Gamma-C$. 
\end{lemma}
\begin{proof}
If condition $(\dagger)$ holds with $N=0$ then clearly $A$ is simply connected at $\infty$. For the converse assume $A$ is simply connected at $\infty$ and 
 $\langle \mathcal A, \mathcal B:T\rangle$ is a presentation for $G$ satisfying the conditions of Definition \ref{SCin}. Define $\Gamma_1\coloneqq \Gamma_{(G,\mathcal A\cup \mathcal B)}(T)$ and $\Gamma_2\coloneqq \Gamma_{(G,\mathcal S)}(R)$. Recall that the vertices of $\Gamma_1$ and of $\Gamma_2$  are both the elements of $G$. In order to avoid confusion if $v$ is a vertex of $\Gamma_1$ we denote by $v'$ the corresponding vertex of $\Gamma_2$. We define proper maps respecting the action of $G$,  $f_1:\Gamma_1\to \Gamma_2$ and $f_2:\Gamma_2\to \Gamma_1$ such that for each vertex $g\in G$ of $\Gamma_1$, $f_1(g)=g'$ and $f_2(g')=g$. If $e$ is an edge of $\Gamma_1$ with initial vertex $v$, terminal vertex $w$ and label $s\in \mathcal A\cup \mathcal B$, then choose an edge path $\tau_s$ in $\Gamma_2$ from $v'\coloneqq f_1(v)$ to $w'\coloneqq f_1(w)$. Define $f_1(e)$ to be $\tau_s$. If $g\in G$ define $f_1$ on $ge$ to be $g\tau_s$. Similarly define $f_2$ from the 1-skeleton of $\Gamma_2$ to the 1-skeleton of $\Gamma_1$. 
Let $M_1$ be the length of the longest path $\tau_s$ for $s\in \mathcal A\cup \mathcal B$ and $M_2$ be the length of the longest path $\tau_{s'}$ for $s'\in \mathcal S$. Note that if $e$ is an edge of $\Gamma_1$, with initial vertex $v$ and terminal vertex $w$, then $f_2f_1(e)$ is an edge path of length $\leq M_1M_2$ from $v$ to $w$, and similarly if $e$ is an edge of $\Gamma_2$. 

In particular, if $x$ is a point of an edge of $\Gamma_1$ then $f_2f_1(x)\in St^{M_1M_2}(x)$. Similarly if $x$ belongs to an edge of $\Gamma_2$. 

If $F$ is a 2-cell of $\Gamma_1$ then the boundary of $F$ is an edge path $\beta_F$ with edge labels the same as an element of $T$. Then $f_1(\beta_F)$ is an edge path loop in $\Gamma_2$. 
Choose $P_1>0$ so that if $F$ is any 2-cell of $\Gamma_1$ then the edge path loop $f_1(\beta_F)$ is homotopically trivial in $St^{P_1}(v')$ for any vertex $v'$ of $f_1(\beta_F)$. The map $f_1$ is defined so that $f_1|_F$ (the restriction of $f_1$ to any 2-cell $F$) realizes this homotopy and respects the action of $G$ on $\Gamma_1$ and $\Gamma_2$. Similarly map the 2-cells of $\Gamma_2$ to $\Gamma_1$ and choose $P_2$ for $f_2$.
Let $L$ be the length of the longest relator of $T\cup R$.

If $x$ is a point of a 2-cell $F$ of $\Gamma_1$ and $v$ is a vertex of $F$, then $f_1(x)\in St^{P_1}(f_1(v))=St^{P_1}(v')$. This means there is an edge path $\tau$ in $\Gamma_2$ of length $\leq P_1$ from $v'$ to a vertex $w'$ and $w'$ belongs to an edge $b$ or 2-cell $B$ containing $f_1(x)$. 

If $w'$ belongs to an edge $b$  then $f_2(f_1)(x)$ belongs to $f_2((\tau,b))$ an edge path of length $\leq M_2( P_1+1)$ that begins at $v$. In this case $f_2(f_1(x))\in St^{M_2(P_1+1)}(x)$. 

Otherwise, $w'$ belongs to a 2-cell $B$ containing $f_1(x)$ and $f_2f_1(x)$ belongs to $St^{P_2}(f_2(w'))=St^{P_2}(w)$.  Then $f_2(\tau)$ is an edge path of length $\leq P_1M_2$ from $v$ to $w$, and $f_2f_1(x)\in St^{P_1M_1+P_2}(v)$. As $x\in St^L(v)$, we have $f_2f_1(x)\in St^{P_1M_1+P_2+L}(x)$. Combining we have: 

\begin{claim}\label{C1} 
There is an integer $M$ such that if $x$ is a point of $\Lambda_1$ (respectively $\Lambda_2$) then $f_2(f_1(x))\in St^M(x)$ (respectively $f_1f_2(x)\in St^M(x)$).
\end{claim}

Let $\Gamma_3\coloneqq \Gamma_{(A,\mathcal A)}$ be the corresponding subgraph of $\Gamma_1$.  Then for any compact set $C$ in $\Gamma_1$ there is a compact set $D$ in $\Gamma_1$ so that any edge path loop in $\Gamma_3-D$, 
is homotopically trivial in $\Gamma_1-C$. Let $\Gamma_4=f_1(\Gamma_3)$. Then $A$ is a subset of the vertices of $\Gamma_4$ and we call these vertices the pseudo vertices of $\Gamma_4$. For each edge $e$ of $\Gamma_3$, $f_1(e)$ is an edge path of $\Gamma_4$ (connecting two pseudo vertices) that we call a pseudo edge of $\Gamma_4$. 

\begin{claim} \label{C2}
Given a compact set $C$ in $\Gamma_2$ there is a compact set $D_1(C)$ in $\Gamma_2$ such that  any pseudo edge path loop $\beta$ in $\Gamma_4-D_1$ is homomtopically trivial in $\Gamma_2-C$. 
\end{claim}
\noindent {\bf Proof:} Assume $C$ is a compact subcomplex of $\Gamma_2$. Then $St^{M+L}(f_2(C))$ is a compact subcomplex of $\Gamma_1$. (See Lemma \ref{bump} for the definition of $L$). As $\Gamma _1$ satisfies Definition \ref{SCin}, there is a compact subcomplex  $E$ of $\Gamma_1$ such that any edge path loop in $\Gamma_3-E$ is homotopically trivial in $\Gamma_1-St^{M+L}(f_2(C))$.

Choose $D_1$ a compact subcomplex of $\Gamma_2$ such that if $w\in G$ is a vertex of $E$ then $f_2(w)\coloneqq w'\in D_1$. 
If $\beta'$ is a pseudo edge path loop in $\Gamma_4-D_1$, let $\beta$ be an edge path loop in $\Gamma_3$ such that $f_1(\beta)=\beta'$. Note that no vertex of $\beta$ belongs to $E$ and so $\beta$ avoids $E$. Then there is a homotopy $H$ that kills $\beta$ in $\Gamma_1-St^{M+L}(f_2(C))$ and $f_1H$ kills $\beta'$ in $\Gamma_2$. It remains to show that the image of $f_1H$ avoids $C$. 
If $im(f_1H)\cap C\ne \emptyset$,  then $im(f_2f_1H)\cap  f_2(C)\ne \emptyset$.  By Claim \ref{C1}, $im(f_2f_1H)\subset St^{M}(im(H))$ and so $St^{M}(im(H))\cap f_2(C)\ne \emptyset$. By Lemma \ref{bump} (1) , $St(im(H))\cap St^{M-1}(f_2(C))\ne\emptyset$ and by Lemma \ref{bump} (2) $im(H)\cap St^{M+L}(f_2(C))\ne \emptyset$.  But $im(H)\cap St^{M+L}(f_2(C))=\emptyset$.
$\square$

Now we complete the proof of Lemma \ref{equiv}. Recall, $N\geq	0$ is an arbitrary fixed integer.
Choose $N_1$ such that if two pseudo vertices of $\Gamma_4$ are within $2N+1$ of one another in $\Gamma_2$ then there is a pseudo edge path of $\Gamma_2$-length $\leq N_1$  connecting them. Let $C$ be compact in $\Gamma_2$. 
Choose $N_2$ so that if $\tau$ is an edge path loop in $\Gamma_2$ of length $\leq N_1+2N+1$, then $\tau$ is homotopically trivial in $St^{N_2}(w')$ for any vertex $w'$ of $\tau$. Now suppose $\alpha$ is an edge path loop of $\Gamma_2-St^{N_2}(D_1(C))$ and each vertex of $\alpha$ is within $N$ of $A$ (the pseudo vertices of $\Gamma_4$). By the definition of $
N_2$,  $\alpha$ is homotopic to a pseudo edge path $\alpha'$ in $\Gamma_4-D_1$ by a homotopy in $\Gamma_2-D_1$. Since $\alpha'$ is homotopically trivial in $\Gamma_2-C$, $\alpha$ is as well.
\end{proof}

\begin{remark} 
Lemma \ref{equiv} implies the following. Suppose the finitely generated group $A$ is simply connected at $\infty$ in the finitely presented group $G$, $(\mathcal S,R)$ is a finite presentation for $G$, and  $v_1,\ldots, v_n$ are vertices of $\Gamma_{(G,\mathcal S)}(R)$. Then for any compact $C\subset \Gamma$ and integer $N\geq 0$ there is a compact set $D(C,N,\{v_0,\ldots ,v_n\})$ such that any loop in $\Gamma-D$, each of whose vertices are within $N$ of $v_iA$ for some $i\in \{1,\ldots ,n\}$, is homotopically trivial in $\Gamma-C$. 
What is not guaranteed is a compact set $D(C,N)$ satisfying the following: For all $v\in G$ and any edge path loop $\alpha$ in $\Gamma-D$ with each vertex of $\alpha$  within $N$ of $vA$, the loop $\alpha$ is homotopically trivial in $\Gamma-C$. 
\end{remark}

\section{Semistability}
For the remainder of the paper, we assume that $G$ is a finitely presented group, $H$ is an infinite finitely generated subgroup of infinite index in $G$ and (as in the statement of Theorem \ref{MainA}) $H$ is subcommensurated in $G$.
 $$H=Q_0\prec Q_1\prec \cdots \prec Q_{k}\prec Q_{k+1}= G.$$
Let $\mathcal H\coloneqq \{h_1,\ldots , h_n\}$ be a finite generating set for $H$, and suppose the group $G$ has generating set $\mathcal G\coloneqq \{h_1,\ldots , h_n, s_1,\ldots , s_m\}$. Let $\mathcal S\coloneqq\{s_1,\ldots, s_m\}$.

The following is Lemma 3.1 of \cite{CM2}.
\begin{lemma}\label{Intersect} 
Suppose $Q$ and $B$ are subgroups of the group $G$ and  $Q\prec G$, then $Q\cap B\prec B$.
\end{lemma}

\begin{lemma}\label{SUB} 

If $H$ is a subgroup of $A$ and $A$ is a subgroup of $G$ ($H<A<G$), then
$$H=Q_0\prec Q_1\cap A\prec \cdots \prec Q_{k}\cap A\prec A.$$
\end{lemma}

\begin{proof}
Recall $Q_{k+1}\coloneqq G$. For $i=1,\ldots, k+1$, define $B_i\coloneqq A\cap Q_i$. As $Q_{i-1}\prec Q_i$ and $B_i<Q_i$, Lemma \ref{Intersect} implies that $Q_{i-1}\cap B_i\prec B_i$. Equivalently, $Q_{i-1}\cap A \prec Q_i\cap A$. 
\end{proof}

\begin{lemma} \label{SUB2} 
Suppose $i\in \{1,2,\ldots, k+1\}$,   
$g\in Q_i $ and $Y$ is a subgroup of $Q_{i-1}$, then $g^{-1}Q_{i-1}g\cap Y$ has finite index in $Y$ and so $gYg^{-1}\cap Q_{i-1}$ 
has finite index in $gYg^{-1}$. Note that if $Y$ is finitely generated, then $g^{-1}Q_{i-1}g\cap Y$ and $Q_{i-1}\cap gYg^{-1}$ are as well.
\end{lemma}
\begin{proof}
The group $g^{-1}Q_{i-1}g\cap Q_{i-1}$ has finite index in $Q_{i-1}$. So, the group $g^{-1}Q_{i-1}g\cap Q_{i-1}\cap Y=g^{-1}Q_{i-1}g\cap Y$ has finite index in $Y$. Conjugating, we have the group  $Q_{i-1}\cap gYg^{-1}$  has finite index in $gYg^{-1}$. 
\end{proof}

For $s\in\mathcal S^{\pm 1}$ let $\mathcal A_s$ be a finite generating set for $Q_{k}\cap s^{-1}Hs$ and define 
$$\mathcal A\coloneqq\mathcal H \cup _{s\in\mathcal S^{\pm 1}}\mathcal A_s.$$ 
Then $A\coloneqq \langle \mathcal A\rangle$ is a finitely generated subgroup of $Q_{k}$.  

The following two lemmas imply the semistability part of Theorem \ref{MainA}.

\begin{lemma} \label{FI}
If $H$ has finite index in $A$, then $H$ is commensurated in $G$ (and so $G$ is 1-ended and semistable at infinity by Theorem \ref{MainCM}). 
\end{lemma}

\begin{lemma}\label{IND} 
If $H$ has infinite index in $A$, then $H$ is subcommensurated in $A$:
$$H=Q_0\prec Q_1\cap A\prec \cdots \prec Q_{k}\cap A=A $$
and both $A$ and $G$ are 1-ended and semistable at infinity.
\end{lemma}
\begin{proof} (of Lemma \ref{FI})
It suffices to show that for $s\in \mathcal S^{\pm 1}$, $s^{-1}Hs\cap H$ has finite index in both $H$ and $s^{-1}Hs$. Since $H$ has finite index in $A$, and $\langle\mathcal A_s\rangle\coloneqq Q_{k}\cap s^{-1}Hs<A$, the group $H\cap (Q_{k}\cap s^{-1}Hs)\coloneqq H\cap s^{-1}Hs$ has finite index in $Q_{k}\cap s^{-1}Hs$. By Lemma \ref{SUB2} (with $Y=H$), the group $Q_{k} \cap s^{-1}Hs$ has finite index in $s^{-1}Hs$ and so $H\cap s^{-1}Hs$ has finite index in $s^{-1}Hs$ for all $s\in \mathcal S^{\pm 1}$. Conjugating we have $sHs^{-1}\cap H$ has finite index in $H$ for all $s\in \mathcal S^{\pm 1}$. Combining we have $s^{-1}Hs\cap H$ has finite index in both $H$ and $s^{-1}Hs$ for all $s\in \mathcal S^{\pm 1}$. 
\end{proof}

\begin{proof} (of Lemma \ref{IND})
Now suppose $H$ has infinite index in $A$. 
The subcommensurated sequence $H=Q_0\prec Q_1\prec \cdots \prec Q_{k}\prec G$ has length $k+1$. Theorem \ref{MainCM} shows that if $k=0$, then $G$ is 1-ended and semistable at $\infty$. Inductively, we assume that if $G'$ is finitely generated and there is a subcommensurated sequence $H'=Q_0'\prec Q_1'\prec \cdots \prec Q_{k-1}'\prec G'$ of length $k$ such that $H'$ is finitely generated and has infinite index in $G'$,  then $G'$ is 1-ended and  semistable at $\infty$. 

In our case, $H$ has infinite index in $A$, and the length $k$ subcommensurated series $H=Q_0\prec Q_1\cap A\prec \cdots \prec Q_{k-1}\cap A\prec A$ implies that $A$ is 1-ended and semistable at $\infty$. Hence we may choose a finite set $P$ of $\mathcal A$-relations so that 
$$\Gamma_{(A,\mathcal A)} (P)\hbox{ is 1-ended and semistable at }\infty.$$ 

If $s\in \mathcal S^{\pm 1}$ and $a\in \mathcal A_s$, then there is a $G$-relation of the form $a=s^{-1} a_ss$ for some $\mathcal H$-word $a_s$. Let $R$ be the (finite) collection of all such relations. 
Define
$$\tilde\Gamma 
\coloneqq \Gamma_{(G,\mathcal A\cup \mathcal S)}(P\cup R).$$ 

We simultaneous show $\tilde \Gamma$ is 1-ended and semistable at $\infty$ by showing all proper edge path rays in $\tilde \Gamma$ are properly homotopic (completing the proof of the semistability part of Theorem \ref{MainA}).

\begin{claim}\label{Cl3} 
Let $K$ be the length of the longest $R$-relation. If $v\in G$ (so $v$ is a vertex of $\tilde \Gamma$), $s\in \mathcal S^{\pm 1}$ and $r$ is a $\mathcal A_s$-proper ray at $v$, then $r$ is properly homotopic $rel \{v\}$ to a ray of the form $(s^{-1}, h_1',h_2',\ldots)$ where $h_i'\in \mathcal H$. Furthermore, this proper homotopy has image in  $St^K(im(r))$.
\end{claim}
\begin{proof}
Suppose $r=(a_1,a_2,\ldots)$ with $a_i\in \mathcal A_s$. Then, $r$ is properly homotopic $rel\{v\}$ to $(s^{-1},(a_1)_s,s, s,^{-1} (a_2)_s^{-1},s, \ldots )$ simply by using the 2-cells for the $R$-relation $a_i=s^{-1}(a_i)_ss$. Then $r$ is properly homotopic $rel \{v\}$ to $(s^{-1},(a_1)_s, (a_2)_s, \ldots )$ by a proper homotopy in $St^K(im(r))$.
\end{proof}

If $v\in G$  and $C_v$ is a compact subcomplex of $v\Gamma_{(A,\mathcal A)}(P)\subset \tilde \Gamma$ then there is a compact subcomplex $D_v$ of $ v\Gamma_{(A,\mathcal A)}(P)$ such that if $r$ and $s$ are edge path rays at $w\in v\Gamma_{(A,\mathcal A)}(P)-D_v$, then, $r$ and $s$ are properly homotopic $rel\{v\}$ by a proper homotopy in  $v\Gamma_{(A,\mathcal A)}-C_v$. 
Hence, if $C$ is a compact subcomplex of $\tilde\Gamma$ and we let $C_v=C\cap v\Gamma_{(A,\mathcal A)}(P)$ (for the finite set of vertices  $v$ such $C\cap v\Gamma_{(A,\mathcal A)}(P)\ne\emptyset$) and let $D=\cup D_v$, then any two $\mathcal A$-rays $r$ and $s$ at $w\in v\Gamma_{(A,\mathcal A)}(P)-D$ are properly homotopic $rel\{w\}$ in $\tilde\Gamma-C$.

By Lemma 2 of \cite{M2} (an elementary graph theory result), for each $v\in G$, there are $\mathcal H$-rays $r_v$ at $v$ such that for any compact set $C\subset \tilde \Gamma$, only finitely many $v$ are such that $r_v$ intersects $C$. Also, for each $s\in \mathcal S^{\pm 1}$, there is an $\mathcal A_s$-ray $r_{(s,v)}$  at $v$ such that for any compact set $C\subset \tilde \Gamma$, only finitely many $v$ are such that $r_{(s,v)}$ intersects $C$.

Choose a sequence of compact subcomplexes $\{C_i\}_{i=1}^\infty$ of $\tilde \Gamma$ satisfying the following conditions:
\begin{enumerate}
\item $\bigcup_{i=1}^\infty C_i= \tilde{\Gamma}$ 
\item $St^K(C_i)$ (see Claim \ref{Cl3}) is contained in the interior of $C_{i+1}$, and the finite set of vertices $v$ such that $r_v$ or $r_{(s,v)}$  ($s\in \mathcal S^{\pm 1}$)  intersects $C_i$, is a subset of $C_{i+1}$.  
\item If $r$ and $s$ are $\mathcal{A}$-rays in $\tilde \Gamma-C_{i}$ both based at a vertex $v$, then $r$ and $s$ are properly homotopic $rel\{v\}$ by a proper homotopy in $v\Gamma_{(A,\mathcal A)}(P)-C_{i-1}$. (See Theorem \ref{ssequiv} part (4).)
\end{enumerate}

For convenience define $C_i=\emptyset$ for $i<1$ and observe that conditions $(1)$, $(2)$, and $(3)$ (see part (5) of Theorem \ref{ssequiv}) remain valid for all $C_i$. 
The next lemma implies  Lemma \ref{IND} and concludes the proof of the semistability part of Theorem \ref{MainA}. 

\begin{lemma}\label{4.3} If $v$ is a vertex of $\tilde{\Gamma}$, and $t = (e_1, e_2,\ldots )$ is an arbitrary ray at $v$, then $t$ is properly homotopic to $r_v$, $rel\{v\}$.

\end{lemma}
\begin{proof}
Assume that $t$ has consecutive vertices $v=v_0, v_1,\ldots $.  By construction, if $v_j \in C_i - C_{i-1}$, then $r_{v_j}$ avoids $C_{i-1}$. Assume $j$ is the largest integer such that $C_j$ avoids $e_i$. Then $r_{v_{i-1}}$ and $r_{v_i}$ avoid $C_{j-1}$. We will show $r_{v_{i-1}}$ is properly homotopic to $e_i*r_{v_i}$ 
$rel\{v_{i-1}\}$ by a proper homotopy $H_i$ with image avoiding $C_{j-3}$. 

If $e_i \in \mathcal{A} ^{\pm 1}$, this is clear by condition $(3)$ with $H_i$ avoiding $C_{j-2}$. 
If $e_i \in {\mathcal S}^{\pm 1}$, then $r_{v_{i-1}}$ and $r_{(e_i,v_{i-1})}$ are $\mathcal A$-rays avoiding $C_{j-1}$ and so by $(3)$ are properly homotopic $rel\{v_{i-1}\}$ by a homotopy avoiding $C_{j-2}$. 
By Claim \ref{Cl3} and condition $(2)$, $r_{(e_i,v_{i-1})}$ is properly homotopic $rel\{v_{i-1}\}$ to a ray $(e_i, h_1',h_2',\ldots )$ where $h_i'\in \mathcal H^{\pm 1}$ and the homotopy avoids $C_{j-2}$. By condition $(3)$, $(h_1',h_2',\ldots )$ is properly homotopic $rel \{v_i\}$ to $r_{v_i}$ by a proper homotopy that avoids $C_{j-3}$. Patch these three proper homotopies together to obtain $H_i$. (See Figure 1.)

\vspace {.5in}
\vbox to 2in{\vspace {-2in} \hspace {-2in}
\includegraphics[scale=1]{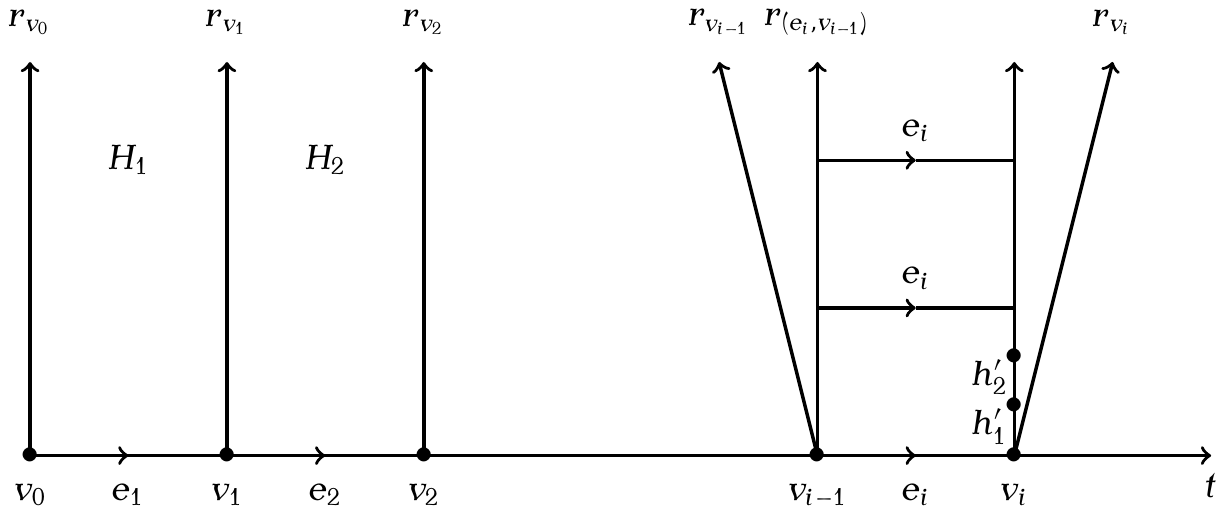}
\vss }

\vspace{-.25in}

\centerline{Figure 1}
\vspace {.25in}

Let $H$ be the homotopy  $rel\{v\}$ of $t$ to $r_v$, obtained by patching together the homotopies $H_i$. 
We need to check that $H$ is proper. Let $C \subset \tilde \Gamma$ be compact. Choose an index $j$ such that $C \subseteq C_j$. Since $t$ is a proper edge path to infinity, choose an index $N$ such that all edges after the $N^{th}$-edge of $t$ avoid $C_{j+3}$. Then for all $i>N$, $H_i$ avoids $C_j$, so $H$ is proper.
\end{proof}
This concludes the proof of Lemma \ref{IND} and the first part of Theorem \ref{MainA}. 
\end{proof}
\section{Simple Connectivity at $\infty$}

It is straightforward to check that the proof of the simply connectivity at $\infty$ part of Theorem \ref{MainCM} given in \cite{CM} extends to the finitely generated case (as follows): If $\langle \mathcal S;R\rangle$ is a finite presentation of the group $G$ then $\Gamma_{(G,\mathcal S)}(R)$ is simply connected. The only time the simple connectivity of $\Gamma$ is used in the proof of Theorem \ref{MainCM} is via the fact: 

$(\ast)$ If $C$ is a compact subset of $\Gamma$ and $N$ is a fixed positive integer, then there is an integer $M(N,C)$ such that any edge path loop $\alpha$ of length $\leq N$ in $\Gamma-St^M(C)$ is homotopically trivial  in $\Gamma-C$. 

Suppose $G$ is a finitely generated subgroup of a finitely presented group $W$ and $W$ has presentation $\langle \mathcal W;R\rangle$ where $\mathcal W$ contains a set of generators $\mathcal G$ for $G$. When proving a finitely generated version of the simply connected at infinity part of Theorem \ref{MainCM}, all work is done in the simply connected space $\Gamma_{(W,\mathcal W)}(R)$, and  one only needs $(\ast)$ for edge path loops $\alpha$ with edge labels in $\mathcal G^{\pm 1}$. Hence, the proof of the simply connected at $\infty$ part of Theorem \ref{MainCM} directly extends to the stronger finitely generated version: 

\begin{theorem} (G. Conner, M. Mihalik Improved)\label{MainFG} 
Suppose $H$ is a 1-ended, finitely presented infinite subgroup of infinite index in the finitely generated group $G$, and  $H$ is commensurated in $G$.
Then $G$ is 1-ended and simply connected at $\infty$. 
\end{theorem}  

In order to finish the proof of our main theorem it remains to prove:
 
\begin{theorem} \label{SubC} 
Suppose that $H$ is a 1-ended, finitely presented, subcommensurated subgroup of infinite index in the finitely generated group $G$:
$$H=Q_0\prec Q_1\prec \cdots \prec Q_{k}\prec Q_{k+1}= G$$
Then $G$ is simply connected at infinity.
\end{theorem}

\noindent {\bf Proof:} We say $H$ is $(k+1)$-subcommensurated in $G$. When $k=0$,  Theorem \ref{MainFG} implies that $G$ is simply connected at $\infty$. 
Assume (inductively) the statement of Theorem \ref{SubC} is valid when $H$ is $(n+1)$-subcommensurated for $n<k$. Let $$\mathcal H\coloneqq \{h_1,\ldots , h_n\}\hbox{ generate }H$$  
$$\mathcal G\coloneqq \{h_1,\ldots ,h_n,s_1,\ldots ,s_m\}\hbox{ generate }G \hbox{ and }$$   
$$\hbox{ let }\mathcal S=\{s_1,\ldots ,s_m\}.$$ 


For $p$ an element of a group $P$ with generating set $\mathcal P$, let $|g|_{\mathcal P}$ be the smallest integer $\ell$ such that $g$ is a product of $\ell$ elements in $\mathcal P^{\pm 1}$. We use the following notation : 
$|g|\coloneqq |g|_{\mathcal G}$ for all $g\in G$.

For each $s\in \mathcal S^{\pm}$ let $\mathcal A_s$ be a finite generating set for $sHs^{-1}\cap Q_{k}$ (see Lemma \ref{SUB2} with $Y=H$) and let $\mathcal A_s'\coloneqq s^{-1}\mathcal A_s s\subset H$.  
Choose an integer $L_1$ such that 
$$L_1\geq |a|_{\mathcal H}\hbox{ for all }a\in (\cup _{s\in \mathcal S^{\pm 1}}\mathcal A'_s).$$ 
We have: 
$$A_s\coloneqq \langle \mathcal A_s\rangle<Q_k\hbox{ has finite index in }sHs^{-1} \hbox{ and }$$ $$A'_s\coloneqq \langle \mathcal A'_s\rangle =s^{-1}A_ss<H \hbox{ has finite index in }H.$$ 

As in $\S 3$ define $$\mathcal A_1\coloneqq \cup _{s\in \mathcal S^{\pm 1}} \mathcal A_s,\ \mathcal A\coloneqq\mathcal H\cup \mathcal A_1\hbox{ and }A\coloneqq\langle \mathcal  A\rangle <Q_k.$$

For each $s\in \mathcal H^{\pm 1}$ and $a\in \mathcal A_s$ there is an $\mathcal H$-word $w(s,a)$ of length $\leq L_1$, such that $s^{-1}asw^{-1}(a,s)$ is a $(A,\mathcal A)$-relator, which we denote by $r(a,s)$. Let
$$R_1=\{r(a,s):s\in \mathcal H^{\pm 1}, a\in \mathcal A_s\}.$$ 
For each $g\in G$ let $\mathcal B_g$ be a finite generating set for the group $gAg^{-1}\cap Q_{k}$. 
(See Lemma \ref{SUB2} with $Y=A$.) 
Let $\mathcal B_g'\coloneqq g^{-1}\mathcal B_g g\subset A\cap g^{-1}Q_kg.$ 
If $g\in A$ then $gAg^{-1}=A$ and so we define $\mathcal B_g\coloneqq \mathcal A\coloneqq \mathcal B_g'$.

 Then 
$$B_g\coloneqq \langle \mathcal B_g\rangle=gAg^{-1}\cap Q_k\hbox{ has finite index in }gAg^{-1} \hbox{ and }$$ $$B'_g\coloneqq \langle \mathcal B'_g\rangle =g^{-1}B_gg=A\cap g^{-1}Q_kg \hbox{ has finite index in }A.$$ 

For each $g\in G$, let $N_g$ be an integer so that in the Cayley graph $\Gamma _{(G,\mathcal G)}$, each vertex of $A$ is within $N_g$ of a vertex of $B'_g$.
Let $$\mathcal B_j=\mathcal A \cup (\cup _{\{g\in G:|g|\leq j\}}  \mathcal B_g)\subset Q_k\hbox{ and }N_j\coloneqq max\{N_g:g\in G\hbox{ and }|g|\leq j\}.$$ 

\begin{lemma}\label{close} 
Suppose $g\in G$ and $y\in gA$. Then in $\Gamma_{(G,\mathcal G)}$, $y$ is within $N_g+|g|$ of a point of $B_g$. 
\end{lemma}
\begin{proof} 
Let $y=ga$ for some $a\in A$. There is $b'\in B_g'\coloneqq g^{-1}Q_kg\cap A$ within $N_g$ of $a$. Then $y'\coloneqq gb'$ is within $N_g$ of $y=ga$. As $y'g^{-1}=gb'g^{-1} \in Q_k\cap gAg^{-1}= B_g$, $y'$ is within $|g|$ of $B_g$ and so $y$ is within $N_g+|g|$ of $B_g$.
\end{proof}

If $H$ has finite index in $A$, then by Lemma \ref{FI}, $H$ is commensurated in $G$ and so $G$ is simply connected at infinity by Theorem \ref{MainFG}. So, we may assume that $H$ has infinite index in $A$. Our induction hypothesis, Lemma \ref{SUB} and the results of $\S 3$ imply:

\begin{lemma} \label{SC} 
The finitely generated subgroups $A$ and $B_j$ of $Q_{k}$, are 1-ended, semistable at $\infty$ and simply connected at $\infty$ for all $j\geq 1$.
\end{lemma} 

Next assume that $G$ is a subgroup of a finitely presented (over) group $W$. Then  for all $j\geq 1$,  $A$ and $B_j$ are simply connected at $\infty$ in $W$. Let $\mathcal W$ be a finite generating set for $W$ containing 
$\mathcal A$ and 
$\mathcal G$, and let $\langle \mathcal W; R\rangle$ be a finite presentation for $W$. Assume that $R$ contains a set $R'$ of $\mathcal A$-relations so that $\Gamma_{(A,\mathcal A)}(R')$ is semistable at $\infty$. We also assume that $R'$ contains the set of conjugation relations $R_1$. If $v$ is a $G$-vertex of $\Gamma_{(W,\mathcal W)}(R)$, let $v\Gamma_{(A,\mathcal A)}(R')$ be the copy of $\Gamma_{(A,\mathcal A)}(R')$ at $v$. To ease notation, if $p$ is a $\mathcal G^{\pm 1}$-word and $\bar p$ is the corresponding element of $G$, define  
$B_p\coloneqq B_{\bar p}$ and $B'_p\coloneqq B'_{\bar p}$.

As a direct consequence of Lemma \ref{close} we have:
\begin{lemma}\label{Ai2} 
Suppose $v$ is a $G$-vertex of $\Gamma_{(W,\mathcal W)}(R)$ and $(e_1,\ldots , e_i)$ labels a $\mathcal G$-edge path with consecutive vertices $v\coloneqq v_0,v_1,\ldots , v_i$. If $j\in \{1,2,\ldots ,i\}$ and $w$ is a vertex of the Cayley graph $v_j\Gamma_{(A,\mathcal A)}$ then  $w$ is within $j+N_j$ of a vertex of $v\Gamma_{(B_{j}, \mathcal B_{j})}$. 
\end{lemma}

Note that Theorem \ref{SubC} does not follow directly from Lemma \ref{equiv} and the fact that $B_{\ell}$ is simply connected at $\infty$ in $W$ for all $\ell$, but Theorem \ref{SubC} does follow from the next lemma:

\begin{lemma}\label{ML} 
Given any compact subcomplex $C$ of $\Gamma_{(W,\mathcal W)}(R)$ there is a compact subcomplex $D$ of $\Gamma_{(W,\mathcal W)}(R)$ such that any $\mathcal G$-loop $\alpha$ at a $G$-vertex of $\Gamma_{(W,\mathcal W)}(R)-D$ is homotopically trivial in $\Gamma_{(W,\mathcal W)}(R) -C$. 
\end{lemma}
\begin{proof} If $v$ is a vertex  of $\Gamma_{(W,\mathcal W)}(R)$ and $v\Gamma _{(A,\mathcal A)}(R')\cap C=\emptyset$, then by part 5) of Theorem \ref{ssequiv}, any two $A$-rays at $v$ are properly homotopic relative to $v$ by a proper homotopy in $v\Gamma _{(A,\mathcal A)}(R')$ (and so the homotopy avoids $C$). There are only finitely many $v\Gamma _{(A,\mathcal A)}(R')$ that intersect $C$. 
Since $\Gamma_{(A,\mathcal A)}(R')$ is semistable at $\infty$ there is  a compact subcomplex $D_1$  of $\Gamma_{(W,\mathcal W)}(R)$ such that any two $\mathcal A$-edge path rays based at a $G$-vertex $v$ and with image in $\Gamma_{(W,\mathcal W)}(R)-D_1$, are properly homotopic relative to $v$ by a homotopy in $v\Gamma_{(A,\mathcal A)}(R')-C$. There are only finitely many Cayley graphs of the form $v\Gamma_{(H,\mathcal H)}$ or $v\Gamma_{(A_s,\mathcal A_s)}$ (for $v\in G$ and $s\in \mathcal S^{\pm 1}$) that intersect $D_1$. Choose a finite subcomplex $D$ of $\Gamma _{(W,\mathcal W)}(R)$ such that $D$ contains the bounded components of both $v\Gamma_{(A_s,\mathcal A_s)}-St^{L_1}(D_1)$ and $v\Gamma_{(H,\mathcal H)}-D_1$, for all $v\Gamma_{(A_s,\mathcal A_s)}$ and $v\Gamma_{(H,\mathcal H)}$ that intersect $D_1$. (Recall, if $s\in \mathcal S^{\pm 1}$ and $a\in \mathcal A_s$, then there is a $\mathcal H$-word $w(a,s)$ of length $\leq L_1$, such that $s^{-1}asw(a,s)^{-1}\in R_1$.)

\vspace {.5in}
\vbox to 2in{\vspace {-2in} \hspace {-1.75in}
\includegraphics[scale=1]{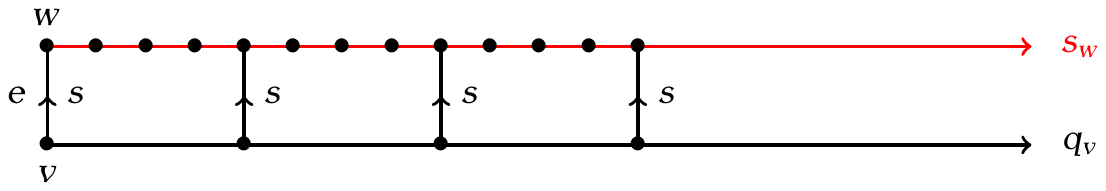}
\vss }

\vspace{-1.5in}

\centerline{Figure 2}
\vspace {.3in}

$(\ast)$ So, if $e$ is an edge of $\Gamma_{(W,\mathcal W)}(R)-D$ with initial vertex $v\in G$, terminal vertex $w$ and label $s\in \mathcal S^{\pm 1}$, then there is a proper $A_s$-ray $q_v$ at $v$ avoiding $St^{L_1}(D_1)$ and hence a $\mathcal H$-ray  $s_w$ at $w$ avoiding $D_1$, such that $q_w$  and $(e,s_v)$ are homotopic relative to $v$ by a homotopy (using only 2-cells arising from $R_1$-conjugation relations) in $\Gamma_{(W,\mathcal W)}(R)-C$ (see Figure 2) and;

$(\ast \ast)$ if $v$ is a $G$-vertex of $\Gamma_{(W,\mathcal W)}(R)-D$ then there is a $\mathcal H$-ray at $v$ in $\Gamma_{(W,\mathcal W)}(R)-D_1$.

Assume $\alpha$ is a $\mathcal G$-loop based at the $G$-vertex $v$ in $\Gamma_{(W,\mathcal W)}(R)-D$. We wish to show that $\alpha$ is homotopically trivial in $\Gamma_{W,\mathcal W)}-C$. Since $G$ is 1-ended, we may assume that $\Gamma_{(G,\mathcal G)}-D$ is connected and so there is an edge path in $\Gamma_{(G,\mathcal G)}-D$ from $v$ to a vertex of $H$. Hence we assume, without loss, that $v\in H$. 
Let $\ell$ be the length of $\alpha$.  By Lemma \ref{SC}, $B_\ell$ is simply connected at $\infty$ and Lemma \ref{equiv} implies there is a compact subcomplex $E(C, \ell+N_\ell)$ of $\Gamma _{(W,\mathcal W)}(R)$ with the following property: If $\beta$ is an  
edge path loop with image in $\Gamma _{(W,\mathcal W)}(R)-E$ and each vertex of $\beta$ is within $\ell+N_\ell$ of $B_\ell$ then $\beta$ is homotopically trivial in $\Gamma _{(W,\mathcal W)}(R)-C$. 

It is enough to show that $\alpha$ is homotopic to such a $\beta$ in $\Gamma _{(W,\mathcal W)}(R)-C$ (and this is where the semistability of $A$ comes in).
Let $s_{v}$ be a $\mathcal H$-proper edge path ray at $v$ in $\Gamma_{(W,\mathcal W)}(R)-D_1$ (see $(\ast,\ast)$). If all edges of $\alpha$ are $\mathcal H$-edges, then by the definition of $D_1$, the rays $s_{v_0}$ and $(\alpha, s_{v_0})$ are properly homotopic relative to $v$, by  $\hat H$ a proper homotopy with image in $v\Gamma_{(A,\mathcal A)}(R')-C$.   Otherwise, write $\alpha$ as $(\tau_0, e_0, \tau_1, e_1,\ldots , \tau_{j}, e_j, \tau_{j+1})$ where $\tau_i$ is a (possibly trivial) $\mathcal H$-path and $e_i$ is a $\mathcal S$-edge. Let the initial vertex of $e_i$ be $v_i'$ and the terminal vertex of $e_i$ be $ v_i$.  (See Figure 3.)

\vspace {.5in}
\vbox to 2in{\vspace {-2in} \hspace {-2.25in}
\includegraphics[scale=1]{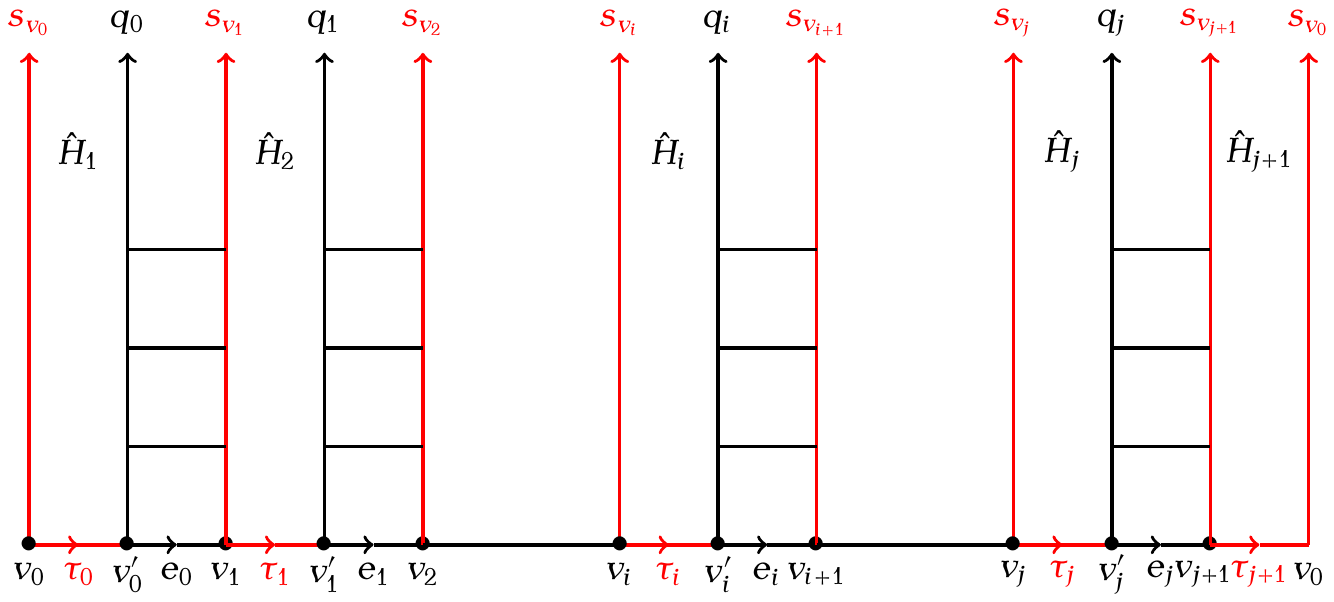}
\vss }

\vspace {.1in}

\centerline{Figure 3}
\vspace {.1in}

If $e_i$ is labeled by $s\in \mathcal S^{\pm 1}$ then let $q_i$ be a proper $\mathcal A_{s}$-ray at $v_i'$, avoiding $St^{L_1}(D_1)$. For each edge $a$ of $q_i$, there is an 2-cell with boundary label $(s^{-1}, a, s, w^{-1}(a,s))$ where $w(a,s)$ is a $\mathcal H$-word of length $\leq L_1$ (see the definition of $R_1$). So if $q_i$ is labeled $(a_1,a_2,\ldots)$ then $s_{v_i}$, the $\mathcal H$-ray at $v_i$, with labeling $(w(a_1,s), w(a_2,s), \ldots)$ 
is such that $q_i$ is properly homotopic to $(e_i,s_{v_i})$ relative to $v_i'$ by a homotopy (only using $R_1$-cells) with image in $\Gamma_{(W,\mathcal W)}(R)-C$. 
The ray $s_{v_i}$ has image in $\Gamma_{(W,\mathcal W)}(R)-D_1$ and so by the semistability of $\Gamma_{(A,\mathcal A)}(R')$ we have for $i\in \{0,1,\ldots, j\}$, $q_i$ is properly homotopic to $(\tau_i^{-1}, s_{v_{i}})$ relative to $v_i'$ by a proper homotopy $\hat H_i$ in $v_i'\Gamma_{(A,\mathcal A)}-C$. Finally, define $\hat H_{j+1}$ to be a proper homotopy in $v_0\Gamma_{(A,\mathcal A)}-C$ of $s_{v_0}$ to $(\tau_{j+1}^{-1}, s_{v_{j+1}})$.

Assume that $v_i'$ is the $j(i)^{th}$ vertex of $\alpha$. 
By Lemma \ref{Ai2}, every vertex of $v_{i}'\Gamma_{(A,\mathcal A)}(R')$ is within $N_{j(i)}+j(i) (\leq N_\ell+\ell)$ of $B_{j(i)}\subset B_\ell$.  Hence for $i\in \{0,\ldots, j\}$, each vertex of (the image of) $\hat H_i$ is within $N_\ell+\ell$ of a vertex of $B_\ell$. 

Combining the homotopies $\hat H_0,\ldots, \hat H_{j+1}$ along with the homotopies of $q_i$ to $(e_i,s_{v_{i+1}})$ we have a proper cellular homotopy $\hat H$ (relative to $v$) of $s_{v_0}$ to $(\alpha, s_{v_0})$ with image in $\Gamma_{(W,\mathcal W)}(R)-C$, and each vertex of the image of the homotopy $\hat H$ is within $N_\ell+\ell$ of a vertex of $B_\ell$. 

$$\hat H:[0,1]\times [0,\infty)\to \Gamma_{(W,\mathcal W)}(R)-C$$ 
such that $\hat H|_{[0,1]\times\{0\}}$ is $\alpha$, $\hat H|_{\{0\}\times [0,\infty)}$ and $\hat H|_{\{1\}\times [0,\infty)}$ both agree with $s_{v_0}$, and each vertex of the image of $\hat H$ is within $N_\ell+\ell$ of $ B_\ell$. Choose $N$ such that $\hat H^{-1}(E)\subset [0,1]\times [0,N]$. The loop $\hat H|_{[0,1]\times\{N+1\}}$ provides a $\mathcal G$-loop $\beta$ such that $\alpha$ is homotopic to $\beta$ in $\Gamma_{(W,\mathcal W)}(R)-C$,  and each vertex of  $\beta$ is within $N_\ell+\ell$ of $B_\ell$. By the definition of $E$, the loop $\beta$ (and hence $\alpha$) is homotopically trivial in $ \Gamma_{(W,\mathcal W)}(R)-C$.
\end{proof}

{\it Michael Mihalik, Department of Mathematics, Vanderbilt University, Nashville TN, 37240 USA.}

{\it email: michael.l.mihalik@vanderbilt.edu}
\end{document}